\documentclass[11pt,leqno]{amsart}
\usepackage{hyperref}
\usepackage{mathrsfs}
\usepackage{tikz-cd}
\usepackage{amsmath,amsfonts,amssymb}
\usepackage{xfrac}
\usepackage{nicefrac}
\usepackage{color}
\usepackage{comment}
\usepackage{booktabs}

\usepackage{todonotes}

\usepackage[margin=1.15in]{geometry}

\theoremstyle{plain}

\newtheorem{theorem}{Theorem}[section]
\newtheorem{lemma}[theorem]{Lemma}
\newtheorem{proposition}[theorem]{Proposition}
\newtheorem{corollary}[theorem]{Corollary}

\theoremstyle{definition}

\newtheorem{remark}[theorem]{Remark}

\numberwithin{equation}{section}

\DeclareMathOperator{\st}{St}

\DeclareMathOperator{\Sym}{Sym}
\DeclareMathOperator{\Aut}{Aut}

\renewcommand{\epsilon}{\varepsilon}

\title[Maximal subgroups of non-torsion GGS-groups]{Maximal
  subgroups of non-torsion Grigorchuk-Gupta-Sidki groups}

\author[D. Francoeur]{Dominik Francoeur} \address{Dominik Francoeur: Instituto de Ciencias Matem\'{a}ticas, Calle Nicol\'{a}s Cabrera, n\textsuperscript{o}13-15, Campus Cantoblanco, Universidad Aut\'{o}noma de Madrid,  28049 Madrid, Spain}
\email{dominik.francoeur@ens-lyon.fr
}

\author[A. Thillaisundaram]{Anitha Thillaisundaram}
 
 \address{Anitha Thillaisundaram: Centre for Mathematical Sciences, Lund University,  
223 62 Lund, Sweden}
 \email{anitha.t@cantab.net}

\date{\today}


 
 \thanks{This research was supported by a London Mathematical Society Research in Pairs (Scheme 4) grant.}
 
\begin{document}

\begin{abstract}
  A Grigorchuk-Gupta-Sidki (GGS-)group is a subgroup of the automorphism group of
  the $p$-regular rooted tree for an odd prime $p$, generated by one rooted automorphism
  and one directed automorphism. Pervova proved that all torsion GGS-groups do
  not have maximal subgroups of infinite index.  Here we extend the result to non-torsion GGS-groups,  which include the weakly regular branch, but not branch, GGS-group. 
\end{abstract}

\maketitle


\section{Introduction}
The automorphism group of an infinite spherically homogeneous rooted tree is well established as a source of interesting finitely generated infinite groups, such as finitely generated groups of intermediate word growth, finitely generated infinite torsion groups, finitely generated amenable but not elementary amenable groups, and finitely generated just infinite groups. Early constructions were produced
by Grigorchuk~\cite{Grigorchuk} and Gupta and Sidki~\cite{Gupta} in the 1980s, which then led to a generalised family of so-called GGS-groups. 

An important type of subgroup of the automorphism group of an infinite spherically homogeneous rooted tree is one having subnormal subgroup structure similar to the corresponding structure in the full group of automorphisms of the tree. These subgroups are termed branch groups; see Section 2 for the definition.

The study of  maximal subgroups of finitely generated branch groups began with the work of Pervova~\cite{Pervova3, Pervova4}, who proved that the torsion Grigorchuk groups and
the torsion GGS-groups do not contain maximal subgroups of infinite index.
Bondarenko~\cite{Bondarenko}  gave
 the first example of a finitely
generated branch group that does have maximal subgroups of infinite
index.  His method does not apply to
groups acting on the  binary and ternary rooted trees. However, recently the first author and Garrido~\cite{FG} provided the first examples of finitely generated branch groups, acting on the binary rooted tree,  with maximal subgroups of infinite index. Their examples are the non-torsion \v{S}uni\'{c} groups. 
For an extensive introduction to the subject of maximal subgroups of finitely generated branch groups, we refer the reader to~\cite{FG}.

The first author  and Garrido (see~\cite{Francoeur}) have further shown that certain non-torsion GGS-groups, namely the generalised Fabrykowski-Gupta groups, one for each odd prime~$p$, do not have maximal subgroups of infinite index. 
In this paper, we extend this result to all non-torsion GGS-groups. 
We recall that a GGS-group acts on the  $p$-regular rooted tree (also called the $p$-adic tree) for $p$ an odd prime, with generators $a$ and $b$, where $a$ cyclically permutes the $p$ maximal subtrees rooted at the first-level vertices, whereas $b$ fixes the first-level vertices pointwise and  is recursively defined by the tuple $(a^{e_1},\dots,a^{e_{p-1}},b)$ which corresponds to the action of~$b$ on the maximal subtrees, for some exponents $e_1,\ldots,e_{p-1}\in\mathbb{F}_p$.

\begin{theorem} \label{thm:main-result} Let $G$ be a GGS-group
  acting on the $p$-regular rooted tree, for $p$ an odd prime. Then every maximal subgroup of $G$  is  of finite index. Furthermore, if~$G$ is branch, then every maximal subgroup is normal and of index~$p$ in~$G$.
\end{theorem}

There is only one GGS-group that is weakly branch but not branch; this is the GGS-group~$\mathcal{G}$ where the automorphism~$b$ is defined by $(a,\ldots,a,b)$. For the group~$\mathcal{G}$, it turns out that there are maximal subgroups that are not normal, nor of index~$p$ in~$\mathcal{G}$. Also there are infinitely many maximal subgroups; see Proposition~\ref{pro:constant}. 

We recall that two
groups are  \textit{commensurable} if they have isomorphic
subgroups of finite index.  The following result is a consequence of the work  of Grigorchuk and Wilson~\cite{GrigWils}, and the proof follows exactly as in~\cite[Cor.~1.3]{AKT}:

\begin{corollary} \label{cor:main}
  Let $G$ be  a branch GGS-group
  acting on the $p$-regular rooted tree, for $p$ an odd prime.  If $H$ is a group
  commensurable with~$G$, then every maximal subgroup of $H$ has finite
  index in~$H$.
\end{corollary}

\smallskip

\noindent\textit{Organisation}. Section~2 contains preliminary material on groups acting on the $p$-adic tree. In Section~3, we formally define the GGS-groups and state some of their basic properties. In Section~4 we set up the scene to prove Theorem~\ref{thm:main-result}, whose proof we complete in Section~5.

\smallskip

\noindent
\textit{Notation.}
Throughout, we  use the convention $[x,y] = x^{-1}y^{-1}xy$ and $x^y=y^{-1}xy$. For a group acting on a rooted tree, we always use the right action.

\medskip

\noindent\textbf{Acknowledgements.} We thank the referees for suggesting valuable improvements to the exposition of the paper.


\section{Preliminaries} \label{sec:2} In the present section we recall
the notion of branch groups and establish prerequisites for the rest
of the paper.  For more information,
see~\cite{BarthGrigSunik,NewHorizons}.

\subsection{The \texorpdfstring{$p$}{p}-regular rooted tree and its automorphisms}
Let $T$ be the  $p$-regular rooted tree, 
 for an odd prime~$p$.  Let
$X=\{1,2,\ldots,p\}$ be an alphabet on $p$ letters.  The
set of vertices of~$T$ can be identified with the free monoid~$X^*$,  and we will freely use this identification without special mention. The root of~$T$ corresponds to the
empty word~$\varnothing$, and for each word $v\in X^*$ and letter~$x$, an edge connects $v$ to $vx$. There is a natural length function
on~$X^*$, and the words~$w$ of length $|w| = n$, representing
vertices that are at distance~$n$ from the root, form the
\textit{$n$th layer} of the tree. 
The
\textit{boundary}~$\partial T$ consisting of all infinite simple rooted paths
is in one-to-one correspondence with the $p$-adic integers.  

For $u$ a vertex, we write $T_u$ for the full rooted subtree of~$T$ that has its root at~$u$ and includes all vertices $v$ with $u$ a prefix of~$v$. For any two vertices $u$ and $v$ the
subtrees $T_u$ and $T_v$ are isomorphic under the map that deletes the
prefix $u$ and replaces it by the prefix $v$. We refer to this
identification as the \emph{natural identification of subtrees} and write
$T_n$ to denote the subtree rooted at a generic vertex of level~$n$. 

We observe that every automorphism of~$T$ fixes the root and that the
orbits of $\Aut T$ on the vertices of the tree~$T$ are precisely its
layers.  Let $f \in \Aut T$ be an automorphism of~$T$.  The image of a
vertex $u$ under $f$ is denoted by $u^f$.  For a vertex $u$, considered as a word over $X$, and a letter $x \in X$ we have $(ux)^f=u^fx'$
where $x' \in X$ is uniquely determined by $u$ and~$f$.  This gives
a permutation $f(u)$ of $X$ so that
\[
(ux)^f = u^f x^{f(u)}.
\]
The permutation~$f(u)$ is called the \textit{label} of~$f$ at~$u$. The automorphism $f$ is called \textit{rooted} if $f(u)=1$ for $u\ne
\varnothing$.  The automorphism~$f$ is called \textit{directed}, with directed path
$\ell$ for some $\ell\in \partial T$, if the support $\{u \mid f(u)\ne1 \}$ of its
labelling is infinite and contains only vertices at distance $1$ from~$\ell$.

The \textit{section} of $f$ at a vertex $u$ is the unique automorphism
$f_u$ of $T \cong T_{|u|}$ given by the condition $(uv)^f = u^f
v^{f_u}$ for $v \in X^*$.

\subsection{Subgroups of \texorpdfstring{$\Aut T$}{Aut T}}
Let $G\le \Aut T$. The
\textit{vertex stabiliser} $\text{st}_G(u)$ is the subgroup consisting of
elements in~$G$ that fix the vertex~$u$.  For $n \in \mathbb{N}$, the
\textit{$n$th level stabiliser} $\st_G(n)= \cap_{|v|=n} \text{st}_G(v)$
is the subgroup  of automorphisms that fix all vertices at
level $n$.  We emphasise the difference in the notation of vertex stabilisers and level stabilisers, which aims to avoid confusion from the fact that the first-level vertices of~$T$ are often identified with the elements of~$X$. Note that elements in $\st_G(n)$ fix all vertices up to
level $n$ and that $\st_G(n)$ has finite index in $G$.

The full automorphism group $\Aut T$ is a profinite group:
\[
\Aut T= \varprojlim_{n\to\infty} \Aut T_{[n]},
\]
where $T_{[n]}$ denotes the subtree of $T$ on the finitely many
vertices up to level~$n$. The topology of $\Aut T$ is defined by the
open subgroups $\st_{\Aut T}(n)$, for $n \in \mathbb{N}$.  
For $G\le \Aut T$, we say that the subgroup $G$ 
has the \textit{congruence subgroup property} if 
for
every subgroup $H$ of finite index in $G$, there exists some $n\in \mathbb{N}$ such
that $\st_G(n)\leq H$. 

For $n\in \mathbb{N}$, every $g\in \st_{\Aut T} (n)$ can be identified with a collection
$g_1,\ldots,g_{p^n}$ of elements of $\Aut T_n$, where $p^n$ is the
number of vertices at level $n$.  Denoting the vertices of $T$ at level $n$ by $u_1, \ldots, u_{p^n}$, there is a natural isomorphism
\[
\st_{\Aut T}(n) \cong \prod\nolimits_{i=1}^{p^n} \Aut T_{u_i}
\cong \Aut T_n \times \overset{p^n}{\cdots} \times \Aut T_n.
\]
Recall that $\Aut T_n$ is isomorphic to $\Aut T$ via the
natural identification of subtrees. Therefore the decomposition
$(g_1,\ldots,g_{p^n})$ of~$g$ defines an embedding
\[
\psi_n \colon \st_{\Aut T}(n) \rightarrow \prod\nolimits_{i=1}^{p^n}
\Aut T_{u_i} \cong \Aut T \times \overset{p^n}{\cdots} \times
\Aut T.
\]
For convenience, we will write $\psi=\psi_1$.

For  $\omega\in X^*$, we further define 
\[
\varphi_\omega :\text{st}_{\Aut T}(\omega) \rightarrow \Aut T_{\omega} \cong \Aut T
\]
to be the natural restriction of $f\in \text{st}_{\Aut T}(\omega)$ to its section $f_\omega$.

We write $G_u=\varphi_u(\text{st}_G(u))$ for the restriction of the vertex stabiliser
$\text{st}_G(u)$ to the subtree rooted at a vertex $u$. 
We say that $G$
is \emph{self-similar} if $G_u$ is contained in~$G$ for every vertex~$u$, and we say that $G$ is \textit{fractal} if $G_u$ equals~$G$ for every vertex~$u$, after the natural identification of subtrees.

The subgroup $\text{rist}_G(u)$, consisting of all automorphisms in~$G$ that fix all vertices~$v$ of~$T$ not having $u$ as a prefix, is
called the \textit{rigid vertex stabiliser} of $u$ in $G$. 
The
\textit{rigid $n$th level stabiliser} is the product
\[
\text{Rist}_G(n)=\prod\nolimits_{i=1}^{p^n} \text{rist}_G(u_i) \trianglelefteq G
\]
of the rigid vertex stabilisers of the vertices $u_1, \ldots, u_{p^n}$
at level~$n$. 

Let $G$ be a subgroup of $\Aut T$ acting \textit{spherically
  transitively}, i.e. transitively on every layer of~$T$. 
Then $G$ is a \emph{weakly branch
  group} if $\mathrm{Rist}_G(n)$ is non-trivial for every
$n \in \mathbb{N}$, and $G$ is a   \emph{branch
  group} if $\mathrm{Rist}_G(n)$ has finite index in $G$ for every
$n \in \mathbb{N}$.  
If, in addition, the group $G$ is
self-similar and $1 \not = K \leq G$ with
$K\times \cdots \times K \subseteq \psi(K\cap \st_G(1))$ and
$\lvert G : K \rvert < \infty$, then $G$ is said to be \emph{regular
  branch over $K$}.  If the condition $\lvert G : K \rvert < \infty$ in the previous definition is omitted, then $G$ is said to be \emph{weakly regular branch over $K$}.  


\section{The GGS-groups} 

By $a$ we denote the rooted automorphism, corresponding to the
$p$-cycle $(1 \, 2 \, \cdots \, p)\in \Sym(p)$, that cyclically
permutes the vertices $u_1, \ldots, u_p$ at the first level.  
Given a  non-zero vector $
\mathbf{e} =(e_{1}, e_{2},\ldots , e_{p-1})\in
(\mathbb{F}_p)^{p-1}$,
we recursively define a  directed automorphism $b \in
\st_{\Aut T}(1)$ via
\[
\psi(b)=(a^{e_{1}}, a^{e_{2}},\ldots,a^{e_{p-1}},b).
\]
We call the subgroup $G=G_{\mathbf{e}}=\langle a, b
\rangle$ of $\Aut T$ the \textit{GGS-group} associated
to the defining vector $\mathbf{e}$.  We observe that $\langle a
\rangle \cong\langle b \rangle \cong C_p$
are cyclic groups of order~$p$. Hence, whenever we refer to an exponent of~$a$ or of~$b$, it goes without saying that it is an element of~$\mathbb{F}_p$.

 A GGS-group $G$ acts spherically transitively on the tree $T$, and every section of every element of $G$ is contained in  $G$. Moreover, a GGS-group  $G$ is fractal.

If
  $G_{\mathbf{e}}=\langle a, b\rangle$ is a GGS-group corresponding to the defining vector
  $\mathbf{e}$, then $G$ is an infinite
  $p$-group if and only if $ \sum\nolimits_{j=1}^{p-1} e_{j}= 0$  in $\mathbb{F}_p$;
  compare \cite{NewHorizons, Vovkivsky}.
  
We also call $G=\langle a,b\rangle$ a \emph{generalised Fabrykowski-Gupta group} if 
\[
\psi(b)=(a,1,\ldots,1,b);
\]
 see \cite{Francoeur}.  By~\cite[Thm.~2.16]{FAZR2}, it follows that a GGS-group whose defining vector contains only one non-trivial element is conjugate in $\Aut T$ to a generalised Fabrykowski-Gupta group.

We write $\mathcal{G}= \langle a,b \rangle$ with $\psi(b)
  = (a,\ldots,a,b)$,
for the GGS-group arising from the constant defining vector~$(1,\ldots,1)$. It is known that $\mathcal{G}$ is weakly regular branch~\cite[Lem.~4.2]{FAZR2} but not branch~\cite[Thm.~3.7]{FAGUA}.

For all other GGS-groups $G \ne \mathcal{G}$, we have from \cite{FAZR2} that $G$ is regular branch over $\gamma_3(G)$. 
Further, from~\cite{FAGUA}, a GGS-group~$G$  has the congruence subgroup property and is just infinite if and only if $G \ne \mathcal{G}$; we recall that an infinite group $G$ is
\emph{just infinite} if all its proper quotients are finite.

The following result is useful.
\begin{lemma}\label{lem:derived-product}
For $G$ a GGS-group, let $g\in G'$ and write $\psi(g)=(g_1,\ldots,g_p)$. Then $g_1g_2\cdots g_{p}\in G'$.
\end{lemma}
\begin{proof}
It suffices to show the statement for $g=[a,b]$. As $[a,b]=(b^{-1})^ab$, we have that $\psi([a,b])=(b^{-1}a^{e_1},a^{e_2-e_1},\ldots, a^{e_{p-1}-e_{p-2}},a^{-e_{p-1}}b)$, and the result follows.
\end{proof}

As evident in the proof of the above result, all actions of group elements on the tree~$T$ will be taken on the right.


\subsection{Length function} 
Here we recall the following  items from~\cite{Pervova4}:
the abelianisation $G/G'$ of a GGS-group~$G$ and a natural length function on elements of~$G$.

Let $G =\langle a,b\rangle$ be a
GGS-group acting on the  $p$-adic  tree
$T$.  
We consider
\[
H = \langle \hat{a}, \hat{b}\mid \\ \hat{a}^p=\hat{b}^p=1
 \rangle,
\]
the free product $\langle \hat{a}\rangle * \langle
\hat{b} \rangle$ of  cyclic groups $\langle
\hat{a}\rangle \cong C_p$ and  $\langle
\hat{b} \rangle \cong C_p$.  There is a unique
epimorphism $\pi\colon H \rightarrow G$ such that $\hat{a} \mapsto a$
and $\hat{b} \mapsto b$, which induces an
isomorphism from $H/H'\cong \langle \hat a \rangle \times \langle
\hat{b} \rangle \cong C_p^{2}$ to $G/G'$; see \cite{Pervova4}.

Let $h \in H$.  Recall that $h$ can be uniquely represented in
the form
\[
h = \hat{a}^{\alpha_1} \cdot
 \hat{b}^{\beta_{1}} \cdot
  \hat{a}^{\alpha_2} \cdot \hat{b}^{\beta_{2}}\cdots  \hat{a}^{\alpha_m} \cdot
  \hat{b}^{\beta_{m}} \cdot
  \hat{a}^{\alpha_{m+1}},
\]
where $m\in \mathbb{N}\cup \{0\}$, $\alpha_1, \ldots, \alpha_{m+1}\in \mathbb{F}_p$ with $
  \alpha_i  \ne 0$ for $i\in
    \{2,\ldots,m\}$, and $\beta_{1}, \ldots, \beta_{m} \in \mathbb{F}_p^*$.
    
We denote by $|h| = m$ the \textit{length} of $h$, with
respect to the factor $\langle \hat{b} \rangle$.  Clearly,
for $h_1,h_2 \in H$ we have
\begin{equation} \label{eq:product-h1-h2}
  |h_1 h_2| \leq |h_1| + |h_2|.
\end{equation}

Then, for $G=\langle a,b \rangle$  a GGS-group, 
the \textit{length} of $g \in G$ is
\[
|g| = \min \{ |h| \mid h \in \pi^{-1}(g) \}.
\]
Based on \eqref{eq:product-h1-h2}, we deduce that for $g_1,g_2
\in G$,
\begin{equation} \label{eq:product-g1-g2}
  |g_1 g_2| \leq |g_1| + |g_2|.
\end{equation}

We end this section with the following result from~\cite{AKT}, specialised to the setting of GGS-groups:
\begin{lemma} \cite[Lem.~4.4]{AKT} \label{shortening} Let $G$ be a  GGS-group, and $g \in \st_G(1)$ with $\psi(g)=(g_1,\ldots,g_{p})$.
  Then $\sum_{j=1}^{p} |g_j| \le |g|$, and
  $|g_j| \le  \frac{|g|+1}{2}$ for each
  $j\in \{1,\ldots,p\}$.

  In particular, if $|g|>1$ then $|g_j| < |g|$
  for every $j \in \{1,\ldots,p\}$.
\end{lemma}

\subsection{The GGS-group defined by the constant vector}

Due to the group $\mathcal{G}$ not being branch, some further properties of~$\mathcal{G}$ need to be established for the next section.

By~\cite[Lem.~4.2]{FAZR2}, the group $\mathcal{G}$ is weakly regular branch over $K'$, where $K=\langle (ba^{-1})^{a^i}\mid i\in  \{0,\ldots,p-1\}\rangle^\mathcal{G}$.

\begin{lemma}
Let $\mathcal{G}$ and $K$ be as above. Then $K'$ is a subdirect product of $K\times \overset{p}\cdots\times K$ and $\psi(K'')\ge \gamma_3(K)\times \overset{p}\cdots \times \gamma_3(K)$.
\end{lemma}

\begin{proof}
This is similar to \cite[Proof of Prop.~17]{GUA2}. We first show that $K'$ is a subdirect product of $K\times \overset{p}\cdots\times K$. Indeed,  writing 
$y_i=(ba^{-1})^{a^i}$ for $0\le i\le p-1$, we have for $0\le i< j\le p-1$, that
\begin{align*}
\psi([y_i,y_j])&=\psi((b^{-1})^{a^{i-1}} (b^{-1})^{a^{j-2}} b^{a^{i-2}}b^{a^{j-1}})\\
&=\begin{cases}
(1,\overset{i-3}\ldots,1,a^{-2}ba ,b^{-1}a,1,\overset{j-i-2}\ldots,1 ,a^{-1}b^{-1}a^2,a^{-1}b,1,\ldots,1) & \,\,\,\text{if }j>i+1,\\
(1,\overset{i-3}\ldots,1,a^{-2}ba ,b^{-2}a^2,a^{-1}b,1,\ldots,1) & \,\,\,\text{if }j=i+1,
\end{cases}\\
&=\begin{cases}
(1,\overset{i-3}\ldots,1,y_2 ,y_1^{-1},1,\overset{j-i-2}\ldots,1 ,y_2^{-1},y_1,1,\ldots,1) & \qquad\qquad\qquad\text{if }j>i+1,\\
(1,\overset{i-3}\ldots,1,y_2 ,(y_0^{-1}y_1^{-1})^a,y_1,1,\ldots,1) & \qquad\qquad\qquad\text{if }j=i+1,
\end{cases}
\end{align*}
hence the result.

Thus, for every $k_1\in K$ there is some $g_1\in K'$ with $\psi(g_1)=(k_1,*,\ldots,*)$. As $\mathcal{G}$ is weakly regular branch over $K'$, for every $k_2\in K'$, there is some $g_2\in K'$ with $\psi(g_2)=(k_2,1,\ldots,1)$. Therefore,
$\psi([g_1,g_2])=([k_1,k_2],1,\ldots,1)$,
and the second part of the lemma follows.
\end{proof}

\begin{proposition}\label{pro:virtually-nilpotent}
Let $\mathcal{G}$ be the GGS-group defined by the constant vector. Then every proper quotient of~$\mathcal{G}$ is virtually nilpotent and has maximal subgroups only of finite index.
\end{proposition}

\begin{proof}
 For the first statement, using~\cite[Thm.~4.10]{Francoeur-paper}, it suffices to show that $\mathcal{G}/K''$ is virtually nilpotent. As $\mathcal{G}'/K''$ is of finite index in $\mathcal{G}/K''$, it suffices to show that $\mathcal{G}'/K''$ is nilpotent. Since $\psi(\mathcal{G}')\le K\times \overset{p}\cdots\times K$ by \cite[Lem.~4.2(iii)]{FAZR2}, the result follows from the second part of the  previous lemma. 
 
 The second statement is immediate, because every virtually nilpotent group has maximal subgroups only of finite index: indeed,   a  nilpotent group $G$ has $\Phi(G)$ containing $G'$, see~\cite{Zassenhaus}. Then every maximal subgroup of a nilpotent group $G$ is normal and hence of finite index. As the property of having only maximal subgroups of finite index passes to finite extensions~\cite[Cor.~5.1.3]{Francoeur}, we have that virtually nilpotent groups have maximal subgroups only of finite index.
\end{proof}


\section{Prodense subgroups}
In the present section we lay out the strategy for proving Theorem~\ref{thm:main-result}, where it suffices to restrict to non-torsion GGS-groups in light of \cite{Pervova4}.

We recall that a subgroup $H$ of $G$ is \textit{prodense}  if  $G = NH$ for every
non-trivial normal subgroup $N$ of~$G$.  Let $G$ be a finitely generated group such that every proper quotient of~$G$ has maximal subgroups only of finite index. Then by~\cite[Prop.~2.21]{Francoeur-paper},  every maximal
subgroup of infinite index in $G$ is prodense and every proper prodense
subgroup is contained in a maximal subgroup of infinite index.

We now recall a key result concerning prodense subgroups of weakly branch groups. 
\begin{proposition}\cite[Lem.~3.1 and Thm.~3.3]{Francoeur-paper} \label{3.3.3}
  Let $T$ be a spherically homogeneous rooted tree and~$G$ a weakly branch group acting on~$T$. Suppose that every proper quotient of~$G$ has maximal subgroups only of finite index. If $H$ is a prodense subgroup of~$G$, then $H_u$ is a prodense subgroup of $G_u$, for every vertex~$u$ of~$T$. Furthermore, if $M<G$ is a proper subgroup, then $M_u<G_u$ is a proper subgroup, and if $M<G$ is a maximal subgroup of~$G$ of infinite index, then $M_u$ is a maximal subgroup of~$G_u$, for every vertex~$u$ of~$T$.
\end{proposition}

Note that every proper quotient of a branch GGS-group has maximal subgroups only of finite index, since proper quotients of branch groups are virtually abelian; cf.~\cite[Prop.~2.22]{Francoeur-paper}. This, and Proposition~\ref{pro:virtually-nilpotent}, allows us to use the above proposition to show that non-torsion GGS-groups~$G$ do not possess proper prodense subgroups: for a prodense subgroup $M$ of $G$, we will show that there exists a vertex $u$ of $T$ such that $M_u=G$. This then shows that $M$ is not proper.

For $G$ a  non-torsion GGS-group and $M$ a prodense subgroup of $G$, it remains to show that there exists a vertex $u$ of $T$ such that $M_u=G$. The following result was proved in a more general setting in~\cite{AKT}, but it was stated for \emph{dense} instead of prodense subgroups~$M$. However, the proof, and hence the result, still holds for prodense subgroups.

\begin{proposition}\label{proposition: second}\cite[Prop.~5.4]{AKT}
  Let $G= \langle a,b\rangle$ be a  GGS-group that is not conjugate to a generalised Fabrykowski-Gupta group and let $M$ be a prodense subgroup of~$G$.  If $b\in M$, 
  then there
  exists a vertex~$u$ of~$T$
  such that $M_u=G$.
\end{proposition}

It follows from the above result that it suffices to show that $b\in M_u$ for $M$ a prodense subgroup of~$G$ and $u$~a vertex of~$T$. This will be done in the next section.


\section{Obtaining \texorpdfstring{$b$}{b}.}

For $G=\langle a,b\rangle$ a non-torsion GGS-group with defining vector $\mathbf{e}=(e_1,\ldots,e_{p-1})$, we write $\lambda:=\sum_{i=1}^{p-1} e_i\ne 0$. Further, we write $g\equiv_{G'} h$ when $gh^{-1}\in G'$, for elements $g,h\in G$.

\begin{lemma}\label{general-split-cases}
Let $G=\langle a,b\rangle$ be a non-torsion GGS-group and let $g\in G$ be such that $g\equiv_{G'} b^{t}$ for some $t\ne 0$. For  $u\in X$, write $\varphi_u(g)\equiv_{G'}a^{n_u}b^{m_u}$ for some $n_u,m_u\in \mathbb{F}_p$.  Then, exactly one of the following cases is true.
\begin{enumerate}
\item 
There exists $u\in X$ such that $m_u\ne 0$ but $n_u=0$.
\item For all $u\in X$, if $m_u\ne 0$, then $n_u\ne 0$. Furthermore, there exist at least two vertices $u_1,u_2\in X$ such that $n_{u_k}\ne \lambda m_{u_k}$  for $k\in\{1,2\}$, and there exist at least two vertices $v_1,v_2\in X$ such that $m_{v_1},m_{v_2}\ne 0$.
\end{enumerate}
\end{lemma}

\begin{proof}
As $g\equiv_{{G}'} b^{t}$, it follows that 
\[
g=(b^{j_1})^{a^{l_1}}(b^{j_2})^{a^{l_2}}\cdots (b^{j_n})^{a^{l_n}}
\]
for some $n\in \mathbb{N}$, $l_1,\ldots,l_n\in \mathbb{F}_p$ with $l_k\ne l_{k+1}$ for $1\le k\le n-1$, and $j_1,\ldots,j_n\in \mathbb{F}_p^*$ with $\sum_{k=1}^n j_k= t$  in $\mathbb{F}_p$.

Now, for $u\in X$, we have $\varphi_u(g) = \varphi_u(b^{a^{l_1}})^{j_1}\varphi_u(b^{a^{l_2}})^{j_2}\cdots \varphi_u(b^{a^{l_n}})^{j_n}$, and since
\[\varphi_u(b^{a^{l}}) =
\begin{cases}
a^{e_{u-l}} & \text{ if } l\ne u, \\
b & \text{ if } l=u,
\end{cases}\]
we have
$\varphi_u(g)\equiv_{{G}'} a^{n_u}b^{m_u}$ with
\[
m_u= \sum j_k
\]
where the sum is taken over all $k$ such that $l_k=u$, and
\begin{equation}\label{eq:n-u}
n_u=\sum_{\substack{i=0 \\ i\ne u}}^{p-1} e_{u-i}\,m_i = \sum_{j=1}^{p-1} e_j\,m_{u-j}.
\end{equation}

It is clear that Cases 1 and 2 are mutually exclusive. Thus, it suffices to show that if Case 1 does not hold, then Case 2 does. So we assume now that Case 1 does not hold.
Suppose that $n_u = \lambda m_u$ for all $u\in X$. Then
\[
 \sum_{j=1}^{p-1} e_j\,m_{u-j} -\lambda m_u=0 \quad\text{for all $u\in X$}.
 \]
 This is equivalent to 
\[
\left(\begin{array}{cccc}
e_1 & \cdots & e_{p-1} &-\lambda
\end{array}\right)
\cdot \left(\begin{array}{c}
m_{p}\\
m_{p-1}\\
\vdots \\
m_{2}\\
m_1
\end{array}\right) =\left(\begin{array}{cccc}
e_1 & \cdots & e_{p-1} &-\lambda
\end{array}\right)
\cdot \left(\begin{array}{c}
m_{p-1}\\
m_{p-2}\\
\vdots \\
m_{1}\\
m_{p}
\end{array}\right) =\cdots 
\]
\[
\cdots = \left(\begin{array}{cccc}
e_1 & \cdots & e_{p-1} &-\lambda
\end{array}\right)
\cdot \left(\begin{array}{c}
m_{1}\\
m_{p}\\
\vdots \\
m_{3}\\
m_{2}
\end{array}\right)= 0.
\]
In other words,
\begin{equation}\label{circulant}
 \left(\begin{array}{c}
m_{p}\\
m_{p-1}\\
\vdots \\
m_{2}\\
m_1
\end{array}\right),  \left(\begin{array}{c}
m_{p-1}\\
m_{p-2}\\
\vdots \\
m_1\\
m_{p}
\end{array}\right), \ldots, \left(\begin{array}{c}
m_{1}\\
m_{p}\\
\vdots \\
m_{3}\\
m_{2}
\end{array}\right)
\end{equation}
are all in the subspace orthogonal to $(e_1 , \ldots, e_{p-1} ,-\lambda)$. However the vectors in (\ref{circulant}) form the rows of a circulant matrix. The rank of this circulant matrix is less than $p$ if and only if $\sum_{i=1}^p m_i = 0$; compare \cite[Lem.~2.7(i)]{FAZR2}. However $\sum_{i=1}^p m_i = \sum_{k=1}^n j_k= t\ne 0$. Thus, there is at least one vertex $u_1\in X$ such that $n_{u_1}\ne \lambda m_{u_1}$.

Suppose that there is only one such $u_1$. Then $n_u= \lambda m_u$ for all $u\ne u_1$.
Notice that we have 
\begin{align*}
\sum_{u\in X}n_u &= \sum_{u\in X}\sum_{j=1}^{p-1}e_jm_{u-j} = \sum_{j=1}^{p-1}\sum_{u\in X}e_jm_{u-j}= \sum_{j=1}^{p-1} e_j \sum_{u\in X} m_u=\lambda \sum_{u\in X} m_u.
\end{align*}
Hence, we have
\begin{align*}
n_{u_1}+\lambda \sum_{u\ne u_1} m_u&= \lambda\sum_{u\in X}m_u,
\end{align*}
which yields 
$n_{u_1} = \lambda m_{u_1}$; a contradiction.

For the final statement, clearly $m_u\ne 0$ for at least one $u\in X$, since $\sum_{u\in X} m_u= t\ne 0$. Since we are not in Case 1, it follows that $n_u\ne 0$. From (\ref{eq:n-u}), the result follows.
\end{proof}

\begin{lemma}\label{lemma:SectionLessThanHalf}
 For $G$ a non-torsion GGS-group, let  $x\in G$ be such that $x\equiv_{G'} b^t$ for some $t\ne 0$, and 
write $\psi(x) = (x_1, x_2, \dots, x_p)$. If $x$ is in Case 2 of Lemma~\ref{general-split-cases} and if $|x|=2\mu$ for some $\mu\in \mathbb{N}$, then  there exists $u\in X$ such that  $x_u\ne 1$ and $|x_u|<\mu=\frac{|x|}{2}$.
\end{lemma}

\begin{proof}
Suppose for the sake of contradiction that this is not the case.
By Lemma~\ref{shortening} and using the fact that $x$ is in Case 2 of Lemma~\ref{general-split-cases}, there exist two vertices $v_1,v_2\in X$ such that $|x_{v_1}|=|x_{v_2}|=\frac{|x|}{2}$ and $x_u=1$ for all $u\in X\backslash \{v_1,v_2\}$. This means that
\[
x=\left(b^{i_1}\right)^{a^{v_1}}\left(b^{j_1}\right)^{a^{v_2}}\cdots \left(b^{i_{\mu}}\right)^{a^{v_1}}\left(b^{j_\mu}\right)^{a^{v_2}}
\]
for some $i_1,\dots, i_{\mu}, j_1, \dots, j_{\mu}\in \mathbb{F}_p$. Since the result is true for $x$ if and only if it is true for $x^{a^{-v_1}}$, we can suppose without loss of generality that
\[
x=b^{i_1}\left(b^{j_1}\right)^{a^{v}}\cdots b^{i_{\mu}}\left(b^{j_\mu}\right)^{a^{v}}
\]
for some $v\in X$. Let us set $i=\sum_{k=1}^{\mu}i_k$ and $j=\sum_{k=1}^{\mu}j_k$. Notice that $i+j=t$. Furthermore, since $x$ is in Case 2 of Lemma~\ref{general-split-cases}, there must exist two vertices with non-zero powers of $b$, and our assumptions then imply that $i$ and $j$ are both non-zero.

For all $u\in X\backslash\{p,v\}$, the elements $\varphi_u(b)$ and $\varphi_u(b^{a^v})$ are both powers of $a$, and thus commute. Therefore, for $u\in X\backslash\{p,v\}$, we have
\[
x_u=\varphi_u(x)=\varphi_u(b^i(b^j)^{a^v})=a^{ie_u + je_{u-v}}.
\]

Since
 $x_u=1$ for all $u\in X\backslash \{p,v\}$, we have $ie_u+je_{u-v}=0$ for all $u\in X\backslash \{p,v\}$. In other words, for all $2\leq k \leq p-1$, we have $ie_{kv}+je_{(k-1)v}=0$. By induction, we see that for all $1\leq k \leq p-1$, we have $e_{kv} = r^{k-1}e_{v}$, where $r=-\frac{j}{i}$, which is well-defined, since $i\ne 0$. Notice that $r\ne 0$, since $j\ne 0$, and $r\ne 1$, since $i+j=t\ne 0$. Therefore, we have
\[\lambda = \sum_{k=1}^{p-1}e_k = \sum_{k=1}^{p-1}e_{kv} = \sum_{k=1}^{p-1}r^{k-1}e_v = e_v\sum_{k=0}^{p-2}r^k = e_v\frac{r^{p-1}-1}{r-1} =0,\]
a contradiction. The conclusion follows.
\end{proof}

\begin{lemma}\label{lem:propagates}
Let $G=\langle a,b\rangle$ be a non-torsion GGS-group with $\lambda=\sum_{i=1}^{p-1}e_i$. 
Let $g\in G$ be such that $g\equiv_{G'} a^ib^j$ with $i\ne 0$ and $j\in \mathbb{F}_p$, and write $g=a^i\cdot \psi^{-1}((g_1,g_2,\ldots,g_p))$. Then, for all $u\in X$, we have $\varphi_u(g^p)\equiv_{G'} a^{\lambda j}b^j$, with $|\varphi_u(g^p)|\le \sum_{k=1}^p|g_k|\le |g|$.
\end{lemma}

\begin{proof}
This follows as in \cite[Lem.~7.2.2]{Francoeur}.
\end{proof}

We will now  establish the following in several steps.

\begin{proposition}\label{proposition: first}
Let $G=\langle a,b\rangle$ be a non-torsion GGS-group, and let $H$ be a prodense subgroup. Then there exists $v_0\in X^*$ such that $b\in H_{v_0}$.
\end{proposition}

First, let us consider the set
\[
U=\Big\{g\in \bigsqcup_{v\in X^*} H_v \mid g\equiv_{G'} a^{ \lambda i}b^i \text{ for some }i\ne 0 \Big\}.
\]
Since $H$ is a prodense subgroup of $G$ and $G'$ is a non-trivial normal subgroup in $G$, the set~$U$ is non-empty. Let $y\in U$ be an element of minimal length in $U$. Let $v\in X^*$ be such that $y\in H_v$. By Proposition~\ref{3.3.3}, the subgroup~$H_v$ is  prodense in~$G_v=G$. Hence if we prove the above proposition for $H_v$, we will also have proved it for $H$. Thus, without loss of generality, we may assume that $v$ is the root of the tree~$X^*$. Further, in light of Lemma~\ref{lem:propagates}, we have ``$y$", that is, an element of~$U$ of minimal length, everywhere.

Now let us consider the set
\[
V=\Big\{g\in \bigsqcup_{v\in X^*} H_v \mid g\equiv_{G'} b^i \text{ for some }i\ne 0 \Big\},
\]
and let $x\in V$ be an element of minimal length in $V$. Likewise, such an element exists as $V$ is non-empty by the prodensity of $H$. Let $w\in X^*$ be such that $x\in H_w$. As before, we may assume that $w$ is the root of the tree. Hence, using the fact that we get $y$ everywhere, we have $y$ and $x$ at the root.

If $|x|=1$, then $x=a^{-j}b^ia^j$ with $i\ne 0$ and $j\in \{1,\ldots,p\}$. Thus, we have that $\varphi_j(x)=b^i$. As $i$ is invertible modulo $p$, we obtain $b\in H_j$.

We will now show that $|x|$ must be equal to 1. The case $|x|=0$ is not possible, so we assume that $|x|>1$. In this case, there cannot exist $u\in X$ such that $\varphi_u(x)\equiv_{G'} b^j$ with $j\ne 0$. Indeed, since $|\varphi_u(x)|<|x|$, this would contradict the minimality of~$x$; compare with Lemma~\ref{shortening}. Therefore according to Lemma~\ref{general-split-cases},  there exists $u\in X$ such that $\varphi_u(x)\equiv_{G'} a^{i_u}b^{j_u}$ with  $i_u\ne 0$ and $i_u\ne \lambda j_u$. For this $u\in X$ we write $x_u=\varphi_u(x)$. 

Furthermore, also by Lemma~\ref{general-split-cases}, there exist two vertices $u_1,u_2\in X$ such that $j_{u_1},j_{u_2}\ne 0$ with $i_{u_1},i_{u_2}\ne 0$.  Now by Lemma~\ref{lem:propagates}, we see that $\varphi_v(x_{u_i}^{\,p})\in U$ for $i\in\{1,2\}$ and $v\in X$, and thus
\[
|y|\le |\varphi_v(x_{u_i}^{\,p})|\le |x_{u_i}|.
\]
Hence, using $\sum_{u\in X}|x_u|\le |x|$, it follows that
\begin{equation}\label{eq:2y-x}
2|y|\le |x|.
\end{equation}

\begin{remark} \label{rem:no-a} We observe that if there is a $v\in X^*$ with $a^n\in H_v$ for some $n\ne 0$, then  writing $y\equiv_{G'} a^{\lambda j}b^j$ for some $j\ne 0$, the product of~$y$ with $a^{-\lambda j}$ yields
\[
|x|\le |y|,
\]
which contradicts (\ref{eq:2y-x}), and so we are done in this case. Hence, in what follows, we will assume freely without special mention that $a\not \in H_v$ for all $v\in X^*$. 
\end{remark}

\begin{lemma}\label{minimal-a}
In the set-up above, in particular, assuming that $|x|>1$, let $v\in X^*$ be any vertex and let $z\in H_v$. If $z\ne 1$, then $|z|\geq |y|$.
\end{lemma}

\begin{proof}
Without loss of generality, we may assume that $z\notin \textup{St}_G(1)$. Indeed, otherwise, since $z$ is non-trivial, there must exist some $w\in X^*$ such that $z\in \text{st}_G(w)$ and $\varphi_w(z)\notin \textup{St}_G(1)$. As we have $|z|\geq |\varphi_w(z)|$, it is sufficient to show that $|\varphi_w(z)|\geq |y|$. Thus, we may assume that $z\notin \textup{St}_G(1)$. Hence, we have $z\equiv_{G'} a^{i_z}b^{j_z}$ for some $i_z,j_z\in \mathbb{F}_p$ with $i_z\ne 0$.

Notice that for all $u\in X$, we have
$\varphi_u(z^p)\equiv_{G'} a^{\lambda j_z}b^{j_z}$
by Lemma \ref{lem:propagates}. In particular, if $j_z\neq 0$, by the minimality of $|y|$, we must have $|\varphi_u(z^p)| \geq |y|$. Since Lemma \ref{lem:propagates} also says that $|z|\geq |\varphi_u(z^p)|$, the conclusion follows in this case.

Thus, it only remains to treat the case where $j_z=0$. Let us assume for the sake of contradiction that we have $|z|<|y|$, and let $j_y\ne 0$ be such that $y\equiv_{G'}a^{\lambda j_y}b^{j_y}$. Since $j_z=0$, there exists $k\in \{1,2,\dots, p-1\}$ such that $z^ky\equiv_{G'} b^{j_y}$. By Lemmata~\ref{shortening} and~\ref{lem:propagates} together with (\ref{eq:product-g1-g2}), for all $u\in X$ we have $|\varphi_u(z^ky)|\leq |z|+\frac{|y|+1}{2} < 2|y|$. If there were some $u\in X$ with $\varphi_u(z^ky)\equiv_{G'}b^l$ for some $l\ne 0$,  by the minimality of $|x|$ we would have that $|x|\leq |\varphi_u(z^ky)|<2|y|$, a contradiction to (\ref{eq:2y-x}). Thus, by Lemma~\ref{general-split-cases}, there must exist two vertices $u_1,u_2\in X$ such that $\varphi_{u_l}(z^ky)\equiv_{G'} a^{i_{u_l}}b^{j_{u_l}}$ with $i_{u_l}, j_{u_l}\ne 0$, for $l\in \{1,2\}$.

Let us write $y=a^{i_y} \cdot \psi^{-1}((y_1,\ldots,y_p))$ and $z=a^{i_z} \cdot \psi^{-1}((z_1,\ldots,z_p))$. For any $u\in X$, we have, remembering that we act on the tree~$T$ on the right,
\begin{align*}
\varphi_u(z^ky) &=\varphi_u\Big( \psi^{-1}\big((z_1,\ldots,z_p)^{a^{-i_z}} (z_1,\ldots,z_p)^{a^{-2i_z}} \cdots  (z_1,\ldots,z_p)^{a^{-ki_z }} (y_1,\ldots,y_p)\big)\Big)\\
&= z_{u+i_z}z_{u+2i_z}\cdots z_{u+ki_z}y_u.
\end{align*}
We also have 
\begin{align*}
\varphi_u(z^{k-p}y) &=\varphi_u\Big( \psi^{-1}\big((z_1^{-1},\ldots,z_p^{-1}) (z_1^{-1},\ldots,z_p^{-1})^{a^{i_z}} \cdots  (z_1^{-1},\ldots,z_p^{-1})^{a^{(p-k-1)i_z }} (y_1,\ldots,y_p)\big)\Big)\\
&= z^{-1}_{u}z^{-1}_{u-i_z}\cdots z^{-1}_{u+(k+1)i_z}y_u.
\end{align*}
As $\varphi_u(z^p)\equiv_{G'} 1$, we have 
\[\varphi_u(z^{k-p}y)\equiv_{G'} \varphi_u(z^p)\varphi_u(z^{k-p}y) \equiv_{G'} \varphi_u(z^ky).\]
In particular, this means that for $u\in \{u_1,u_2\}$, we have 
$$
\varphi_{u}(z^{k-p}y)\equiv_{G'} \varphi_{u}(z^{k}y) \equiv_{G'} a^{i_{u}}b^{j_{u}}
$$
with $i_{u}, j_{u}\ne 0$. We have seen above that this implies that $|\varphi_{u}(z^{k-p}y)|, |\varphi_{u}(z^ky)|\geq |y|$.

Since $\sum_{n=1}^{p}|y_n|\leq |y|$, there must exist some $u\in \{u_1,u_2\}$ such that $|y_{u}|\leq \frac{|y|}{2}$. For this $u$, we have
\[
|\varphi_{u}(z^ky)| \leq |y_u|+ \sum_{n=1}^{k}|z_{u+ni_z}|
\]
and
\[
|\varphi_{u}(z^{k-p}y)| \leq |y_u|+  \sum_{n=0}^{p-k-1}|z_{u-ni_z}|.
\]
Since $\sum_{n=0}^{p-1}|z_{u-ni_z}|\leq |z|$, we must have that one of $\sum_{n=1}^{k}|z_{u+ni_z}|$ or $\sum_{n=0}^{p-k-1}|z_{u-ni_z}|$ is at most $\frac{|z|}{2}$. Consequently, we either have
\[
|\varphi_{u}(z^ky)| \leq |y_u| + \frac{|z|}{2} \leq \frac{|y|+|z|}{2}< |y|
\]
or $|\varphi_{u}(z^{k-p}y)|<|y|$. In either case, we get a contradiction, as required.
\end{proof}

The following result is essential.  First let us write $y=a^{i_y} \cdot \psi^{-1}((y_1,\ldots,y_p))\equiv_{G'} a^{\lambda j_y}b^{j_y} $, for some $j_y\ne 0$ and $i_y=\lambda j_y$. By abuse of notation, we will still write $y$ for $\varphi_u(y^p)$ for any $u\in X$; compare with Lemma~\ref{lem:propagates}. Additionally, for notational convenience, we sometimes write $z_v$ or $(z_1,\ldots,z_p)_v$ for $\varphi_v(z)$, where $z=\psi^{-1}((z_1,\ldots,z_p))\in \text{st}_G(v)$ and $v\in X^*\backslash\{\varnothing\}$.

\begin{lemma}\label{w}
In the set-up above, in particular assuming that $|x|>1$, there is an element $w\in H_v$, for some $v\in X^*$, with $w\equiv_{G'} a^nb^m$ for $n\ne \lambda m$, $n\ne 0$ and $|w|=|y|$.
\end{lemma}

\begin{proof}
Let
\[
W=\Big\{g\in \bigsqcup_{v\in X^*} H_v \mid g\equiv_{G'} a^n b^m \text{ for some }m,n \text{ with } n\ne 0, n\ne \lambda m \Big\}
\]
and let $w\in W$ be an element of minimal length. We need to show that $|w|=|y|$. Let us first remark that $|w|\leq \frac{|x|}{2}$. Indeed, since $|x|>1$, it follows from Lemma~\ref{general-split-cases} that there exist $u_1,u_2\in X$ such that $x_{u_1},x_{u_2}\in W$. Since $|x_{u_1}|+|x_{u_2}|\leq |x|$, we conclude that $|w|\leq \min\{|x_{u_1}|, |x_{u_2}|\}\leq \frac{|x|}{2}$.

Let us write $w=a^n\cdot \psi^{-1}((w_1,\ldots,w_p))$. There exists $k\in \{1,2,\ldots,p-1\}$ with $y^k\equiv_{G'} y^{k-p}\equiv_{G'} a^n b^{kj_y}$. Using $y$, or more precisely the projection of some power of $y$ to the vertex where $w$ lives, to cancel $a^n$ in~$w$ yields
\begin{equation}\label{eq:half-length}
\begin{split}
|(y^{-k}w)_v|&= \lvert\big( (y_1^{-1},\ldots,y_p^{-1}) (y_1^{-1},\ldots,y_p^{-1})^{a^{i_y}}  \cdots (y_1^{-1},\ldots,y_p^{-1})^{a^{(k-1)i_y}} (w_1,\ldots,w_p)\big)_v\rvert\\
&=  | y^{-1}_{v}y^{-1}_{v-i_y}\cdots y^{-1}_{v-(k-1)i_y}w_v| \\
&\le |y^{-1}_{v}y^{-1}_{v-i_y}\cdots y^{-1}_{v-(k-1)i_y}|+|w_v|\\
&\le \sum_{d=0}^{k-1} |y_{v-di_y}|+|w_v|
\end{split}
\end{equation}
and
\begin{equation}\label{eq:half-length2}
\begin{split}
|(y^{p-k}w)_v|&= |y_{v-(p-1)i_y}\cdots y_{v-i_y}y_{v} y^{-1}_{v}y^{-1}_{v-i_y}\cdots y^{-1}_{v-(k-1)i_y}w_v| \\
&\le |y_{v-(p-1)i_y}\cdots y_{v-ki_y}|+|w_v|\\
&\le \sum_{d=k}^{p-1} |y_{v-di_y}|+|w_v|
\end{split}
\end{equation}
for $v\in X$.

Since $\sum_{d=0}^{p-1} |y_{v-di}| \leq |y|$ for $v\in X$, we have
\[
\min{\left\{|(y^{-k}w)_v|,|(y^{p-k}w)_v|\right\}} \leq \frac{|y|}{2}+|w_v| \leq \frac{|y|}{2} + \frac{|x|}{4}+\frac{1}{2}\leq\frac{|x|+1}{2}
\]
and
\[
\max{\left\{|(y^{-k}w)_v|,|(y^{p-k}w)_v|\right\}} \leq |y|+|w_v| \leq \frac{|x|}{2} + \frac{|x|}{4}+\frac{1}{2}\leq\frac{3|x|+2}{4},
\]
making use of~\eqref{eq:2y-x}.

As we will see, this implies that there exists a $w_1\in \left\{(y^{-k}w)_v, (y^{p-k}w)_v \mid v\in X\right\}$ such that $w_1\in W$ and $|w_1|\leq \frac{|y|+|w|}{2}$.

Indeed, if $y^{-k}w$ is in Case 1 of Lemma~\ref{general-split-cases}, then there exists $v\in X$ such that $(y^{-k}w)_v \equiv_{G'} b^{t'}$ for some $t'\ne 0$. By the minimality of~$|x|$ and the  above inequalities, we must have $|x|\leq |(y^{-k}w)_v|\leq \frac{3|x|+2}{4}$. This is only possible if $|x|\leq 2$, and since we assume that $|x|>1$, this implies $|x|=2$, and so $|y|=1$ by~\eqref{eq:2y-x}.  Notice that since $|x|=2$, we cannot have $|x|\leq |(y^{-k}w)_v|\leq \frac{|x|+1}{2}$ and thus, by the inequalities above, we must have $|(y^{p-k}w)_v|\leq \frac{|x|+1}{2} = \frac{3}{2}$. Since the length must be an integer, we have  $|(y^{p-k}w)_v| \leq 1 \leq \frac{|y|+|w|}{2}$, where the last inequality follows from Lemma~\ref{minimal-a} applied to $w\ne 1$. Notice that $(y^{p-k}w)_v = (y^p)_v(y^{-k}w)_v \equiv_{G'} a^{\lambda j_y}b^{j_y+t'}$ with $\lambda j_y \ne 0$. Thus, in this case, $w_1=(y^{p-k}w)_v$ has the required properties.

By symmetry, if $y^{p-k}w$ is in Case 1 of Lemma~\ref{general-split-cases}, then we can also find some $w_1$ of the required form. Thus, it only remains to check the case when both $y^{-k}w$ and $y^{p-k}w$ are in Case~2 of Lemma~\ref{general-split-cases}. In this case, there exist at least two vertices $v_1,v_2\in X$ such that $(y^{- k}w)_{v_1}$ and $(y^ {-k}w)_{v_2}$ have non-zero total $a$-exponent differing from $\lambda$ times their total $b$-exponent, and it follows from Lemma \ref{lem:propagates}  that $(y^{p-k}w)_{v_1}$ and $(y^ {p-k}w)_{v_2}$ satisfy the same property. 
Since $|w_{v_1}|+|w_{v_2}|\leq |w|$, we can assume without loss of generality that $|w_{v_1}|\leq \frac{|w|}{2}$. Then, we have
\[\min{\left\{|(y^{-k}w)_{v_1}|,|(y^{p-k}w)_{v_1}|\right\}} \leq \frac{|y|}{2}+|w_{v_1}| \leq \frac{|y|+|w|}{2},\]
and we set $w_1$ as the smallest of these two elements.

We have thus found some $w_1\in W$ with $|w_1|\leq \frac{|y|+|w|}{2}$. By the minimality of $|w|$, we have $|w|\leq |w_1|$, and so $|w|\leq |y|$. As we have $|y|\leq |w|$ from Lemma~\ref{minimal-a}, we conclude that $|w|=|y|$.
\end{proof}

\bigskip

Finally, to prove Proposition~\ref{proposition: first}, we return to our assumption that $|x|>1$. We will now obtain a contradiction using the above results. Hence $|x|=1$, as required.

 We will use  $w$ from Lemma~\ref{w} to show that there
exists a $g \in \bigsqcup_{v\in X^*}H_v$ with $0<|g|<|y|$, which will contradict Lemma~\ref{minimal-a}.

As before, we write  $y=a^{i_y} \cdot \psi^{-1}((y_1,\ldots,y_p))\equiv_{G'} a^{\lambda j_y}b^{j_y} $ for some $j_y\ne 0$ and so $i_y=\lambda j_y$. From Lemma~\ref{w}, we have $w\in H_v$ for some $v\in X^*$ with $w\equiv_{G'} a^nb^m$ for $n\ne \lambda m$, $n\ne 0$ and $|w|=|y|$.
Let $k\in \{1,2,\ldots,p-1\}$ be such that $y^k\equiv_{G'} a^{n}b^{kj_y}$. There must exist at least two vertices $v_1,v_2\in X$ such that $(y^{-k}w)_{v_1}$ and $(y^{-k}w)_{v_2}$ have non-zero total $a$-exponent differing from $\lambda$ times  their total $b$-exponent. Indeed, as we saw in the proof of the previous lemma, both $y^{-k}w$ and $y^{p-k}w$ must be in Case 2 of Lemma~\ref{general-split-cases}, unless possibly when $|x|=2$ and $|y|=1$. As we will now see, this last case is impossible. Certainly, if $|x|=2$, $|y|=1$ and one of $y^{-k}w$ or $y^{p-k}w$ is in Case 1 of Lemma~\ref{general-split-cases}, then there exists  $u\in X$ such that $(y^{-k}w)_u\equiv_{G'} b^{t'}$ or $(y^{p-k}w)_u\equiv_{G'} b^{t'}$ for some $t'\ne 0$. For concreteness, let us assume that $(y^{-k}w)_u\equiv_{G'} b^{t'}$; the other case is similar.
Since $|y|=|w|=1$, there exist $i_1,i_2,i_3,i_4\in \{0,1,\dots, p-1\}$ such that $y=a^{i_1}b^{j_y}a^{i_2}$ and $w=a^{i_3}b^{m}a^{i_4}$ (in particular, notice that we have $i_1+i_2=i_y$ and $i_3+i_4=n$, where these equations are taken modulo $p$). A direct computation then yields that 
\begin{align*}
y^{-k}w &= a^{-i_2}(b^{-j_y})(b^{-j_y})^{a^{i_y}} \cdots (b^{-j_y})^{a^{(k-1)i_y}}a^{i_2-ki_y}a^{i_3}b^ma^{i_4}\\
&= (b^{-j_y})^{a^{i_2}}(b^{-j_y})^{a^{i_2+i_y}} \cdots (b^{-j_y})^{a^{i_2+(k-1)i_y}}(b^m)^{a^{n-i_3}}.
\end{align*}
Therefore, we have
\[(y^{-k}w)_u = (b^{-j_y})_{u-i_2}(b^{-j_y})_{u-i_2-i_y} \cdots (b^{-j_y})_{u-i_2-(k-1)i_y}(b^m)_{u-n+i_3}.\]
Using the fact that for $v\in X$, we have
\[b_v=\begin{cases}
a^{e_v} & \text{ if }v\ne p,\\
b & \text{ if } v=p,
\end{cases}\]
and that $u-i_2, u-i_2-i_y,\dots, u-i_2-(k-1)i_y$ are all different vertices, we see that there are only four different possible forms for $(y^{-kw})_u$: $a^*$, $a^*b^{-j_y}a^*$, $a^*b^{m}$, or $a^*b^{-j_y}a^*b^{m}$, where the stars represent unimportant (possibly zero) powers of $a$. Since we assumed that $(y^{-k}w)_u\equiv_{G'} b^{t'}$, we must have $2=|x|\leq |(y^{-k}w)_u|$, and thus only the last form is possible, since the other forms yield elements of length at most one. In particular, notice that we have $t'=m-j_y$. Let us now consider the element $(y^{p-k}w^{1-p})_u = (y^p)_u(y^{-k}w)_u(w^{-p})_u$. Since $(y^p)_u\equiv_{G'} a^{\lambda j_y}b^{j_y}$ and $w^{-p}\equiv_{G'} a^{-\lambda m}b^{-m}$. we have $(y^{p-k}w^{1-p})_u \equiv_{G'} a^{\lambda(j_y-m)}b^{t'+j_y-m} = a^{-\lambda t'}$. Considerations of length allow us to conclude that in fact, we must have $(y^{p-k}w^{1-p})_u = a^{-\lambda t'}$. Indeed, in exactly the same way as above, we find that
\[(y^{p-k}w^{1-p})_u = \prod_{d=1}^{p-k}(b^{j_y})_{u-i_2+di_y}\prod_{v\in X\backslash \{u-n+i_3\}}(b^{-m})_v.\]
As the sets $\{u-i_2+i_y, u-i_2+2i_y,\dots, u-i_2+(p-k)i_y\}$ and $X\backslash \{u-n+i_3\}$ are disjoint from $\{u-i_2, u-i_2-i_y, \dots, u-i_2-(k-1)i_y\}$ and $\{u-n+i_3\}$, respectively, and that the latter two sets both contain $p$ by previous considerations, we conclude that all the elements in the above product are powers of $a$, and thus that $(y^{p-k}w^{1-p})_u=a^{-\lambda t'}$. As $t'\ne 0$, this element is non-trivial, but it is of length $0$, a contradiction with Lemma \ref{minimal-a}. A similar argument yields a contradiction if we assume that $y^{p-k}w$ is in Case 1 of Lemma~\ref{general-split-cases}. We conclude that both $y^{-k}w$ and $y^{p-k}w$ must be in Case 2 of Lemma~\ref{general-split-cases}.

Let $u\in\{v_1,v_2\}$. We may assume that $|w_u|=\frac{|w|}{2}$, else we are done. Indeed, otherwise, one of $|w_{v_1}|$ or $|w_{v_2}|$ must be strictly smaller than $\frac{|w|}{2}$. Let us suppose without loss of generality that $|w_{v_2}|<\frac{|w|}{2}$. Then either $|(y^{-k}w)_{v_2}|<\frac{|y|}{2}+\frac{|w|}{2}=|y|$ or  $|(y^{p-k}w)_{v_2}|<\frac{|y|}{2}+\frac{|w|}{2}=|y|$. As this contradicts Lemma~\ref{minimal-a}, we have that 
\begin{equation}\label{eq:w-sections}
|w_{v_1}|=|w_{v_2}|=\frac{|w|}{2}
\end{equation}
and hence $|y|=|w|=2\mu$ for some $\mu\in \mathbb{N}$.

By symmetry, that is, by considering $w^{k'}y$ for some $k'$, it likewise follows that only two first-level sections of $y$ are of non-zero length and that the sum of their length must be $|y|$.
Let $i_1,i_2\in X$ be such that 
\[
|y_{i_1}|=|y_{i_2}|=\frac{|y|}{2}=\mu.
\]

If $k=1$, then $|x|\le |y^{-1}w|\leq |y|+|w|=2|y|$ and hence $|y^{-1}w|=2|y|$ by~\eqref{eq:2y-x}. Therefore, it follows from Lemma~\ref{lemma:SectionLessThanHalf} that there exists $u\in X$ such that $(y^{-1}w)_u\ne 1$ and $|(y^{-1}w)_u|<|y|$. As this contradicts Lemma~\ref{minimal-a}, we conclude that the case $k=1$ is impossible. Likewise, we see that the case $k=p-1$ is impossible by looking at $y^{p-(p-1)}w=yw$.

Let us now assume that $1<k<p-1$. For $v\in X$, recall that we have
\[|(y^{-k}w)_v|=|y^{-1}_{v}y^{-1}_{v-i_y}\cdots y^{-1}_{v-(k-1)i_y}w_v|\leq \sum_{d=0}^{k-1} |y_{v-di_y}|+|w_v|\]
and
\[|(y^{p-k}w)_v|=|(y_{v-ki_y}^{-1}y_{v-(k+1)i_y}^{-1}\cdots y_{v-(p-1)i_y}^{-1})^{-1}w_v|\leq \sum_{d=k}^{p-1} |y_{v-di_y}|+|w_v|.\]
Since we assume that $|w_{v_1}|=|w_{v_2}|=\frac{|y|}{2}$, it follows that unless $\sum_{d=0}^{k-1} |y_{v_l-di_y}|=\sum_{d=k}^{p-1} |y_{v_l-di_y}|=\frac{|y|}{2}$ for $l\in \{1,2\}$, then one of $|(y^{-k}w)_{v_l}|$ or $|(y^{p-k}w)_{v_l}|$ is strictly smaller than $|y|$, which is impossible.

Notice also that if there exists $v\in X\backslash \{v_1,v_2\}$ such that $0<|\prod_{d=0}^{k-1}y_{v-di_y}^{-1}|<|y|$, then  $(y^{-k}w)_{v}$ 
is a non-trivial element of length strictly smaller than $|y|$, since we must have $|w_v|=0$ when $v$ is different from $v_1$ or $v_2$. As this contradicts Lemma~\ref{minimal-a}, for all $v\in X\backslash\{v_1,v_2\}$, we must have either $|\prod_{d=0}^{k-1}y_{v-di_y}^{-1}|=0$ or $|\prod_{d=0}^{k-1}y_{v-di_y}^{-1}|=|y|$. Likewise, we must also have either  $|\prod_{d=k}^{p-1}y_{v-di_y}^{-1}|=0$ or $|\prod_{d=k}^{p-1}y_{v-di_y}^{-1}|=|y|$.

This implies strong restrictions on the form of $y$. Recall from above that there exist $i_1,i_2\in X$ such that $|y_{i_1}|=|y_{i_2}|=\frac{|y|}{2}$. This implies that, up to renaming $i_1$ and $i_2$, we must have
\begin{equation}\label{tight-form}
y=a^{i_y}(b^{s_1})^{a^{i_1}}(b^{s_2})^{a^{i_2}}\cdots (b^{s_{2\mu-1}})^{a^{i_1}}(b^{s_{2\mu}})^{a^{i_2}}
\end{equation}
where $s_1+\cdots+s_{2\mu}=j_y$. In particular, for all $u\in X\backslash\{i_1,i_2\}$, we have $|y_u|=0$. This implies that if, for some $v\in X$, the set $\{v-di_y \mid 0\leq d\leq k-1\}$ contains exactly one of either $i_1$ or $i_2$, then $|\prod_{d=0}^{k-1}y_{v-di}^{-1}|=\frac{|y|}{2}$. By the above considerations, there can be only two such sets, namely when $v=v_1$ or $v=v_2$. As the next lemma shows, this can only be the case if $i_2=i_1+i_y$ or $i_2=i_1-i_y$.

\begin{lemma}
Let $i, k, i_1,i_2\in \mathbb{F}_p$ be four elements of $\mathbb{F}_p$, with $i\ne 0$, $i_1\ne i_2$ and $1<k<p-1$. For all $v\in \mathbb{F}_p$, let
\[
I_v = \{v-di\mid 0\leq d \leq k-1\}.
\]
If there exist two elements $v_1, v_2\in \mathbb{F}_p$ such that
\[
|I_{v}\cap\{i_1,i_2\}| = 1 \quad\Longleftrightarrow \quad v\in \{v_1, v_2\},
\]
then either $i_2=i_1+i$ or $i_2=i_1-i$.
\end{lemma}

\begin{proof}
Let $f\colon \mathbb{F}_p\rightarrow\mathbb{F}_p$ be the map defined by $f(x)=i^{-1}(x-i_1)$. As $f$ is a bijection, we have $|f(I_v)\cap \{f(i_1), f(i_2)\}|= |I_{v}\cap\{i_1,i_2\}|$, and since we have
\[f(I_v) = \{f(v)-d \mid 0\leq d \leq k-1\}\]
for all $v\in \mathbb{F}_p$, it suffices to prove the result for $i=1$ and $i_1=0$.

Suppose for the sake of contradiction that $i_2\ne 1, -1$, and let $v_1\in \{k,k+1,\dots, p-1\}$ be the smallest element, with respect to the standard order on $\{k,k+1,\dots, p-1\}$, such that $i_2\in I_{v_1}$. Note that $v_1$ exists, since $\bigcup_{v=k}^{p-1}I_v=\{1,2,\dots, p-1\}$. Note also that since $v_1\in \{k, k+1, \dots, p-1\}$, we know that $0\notin I_{v_1}$. We claim that $i_2\in I_{v_1+1}$ and that $0\notin I_{v_1+1}$. Indeed, we have $I_{v_1+1}\backslash I_{v_1}=\{v_1+1\}$.  Thus, if we had $0\in I_{v_1+1}$, this would imply that $v_1=p-1$. Since $k<p-1$, by the minimality of $v_1$, this would then imply that $i_2\notin I_{p-2}$, and therefore $i_2=p-1$, a contradiction. Similarly, since $I_{v_1}\backslash I_{v_1+1}=\{v_1-k+1\}$, if we had $i_2\notin I_{v_1+1}$, this would imply that $i_2=v_1-k+1$. Since $k>1$, we have $i_2<v_1$, and so $i_2\in I_{v_1-1}$. By the minimality of $v_1$, this implies that $v_1=k$. However, this then means that $i_2=1$, a contradiction.

We conclude that $i_2\in I_{v_1}, I_{v_1+1}$ and that $0\notin  I_{v_1}, I_{v_1+1}$. Likewise, we can also find $v_2\in\{i_2+k,i_2+k+1,\ldots,i_2+p-1\}$ such that $0\in I_{v_2}, I_{v_2+1}$ and $i_2\notin I_{v_2}, I_{v_2+1}$. Clearly $v_1, v_1+1, v_2, v_2+1$ are four different elements, and we have $|I_{v}\cap \{0, i_2\}|=1$ for all $v\in \{v_1, v_1+1, v_2, v_2+1\}$, which contradicts our assumptions. The result follows.
\end{proof}

Coming back to our considerations on the form of $y$, we have
\[y=a^{i_y}(b^{s_1})^{a^{i_1}}(b^{s_2})^{a^{i_2}}\cdots (b^{s_{2\mu-1}})^{a^{i_1}}(b^{s_{2\mu}})^{a^{i_2}},\]
where $i_2=i_1+i_y$ or $i_2=i_1-i_y$ by the previous lemma. If we had $i_2=i_1-i_y$, then we would have
\[\varphi_{i_2-i_y}(y^p) =y_{i_2}y_{i_1} y_{i_2+2i_y}\cdots y_{i_2+(p-1)i_y}.\]
From the form of $y$ above, we see that $y_{i_2}$ ends with $b^{s_{2\mu}}$ and $y_{i_1}$ begins with $b^{s_1}$, which implies that $|y_{i_2}y_{i_1}|\leq |y_{i_2}|+|y_{i_1}|-1$. Consequently, we have $|\varphi_{i_2-i_y}(y^p)|<|y|$, a contradiction to the minimality of $y$ by Lemma~\ref{lem:propagates}. We conclude that we must have $i_2=i_1+i_y$.

To finish the proof, it suffices to show that if $y$ is of the above form, then one of $(y^{p})_{v_1}$ or $(y^p)_{v_2}$ is not. Indeed, we can then repeat the whole argument above with $y'=(y^p)_{v}$ and $w'=(y^{-k}w)_v$ for the corresponding $v\in \{v_1,v_2\}$ and we will reach a contradiction.

Let us then prove this last claim. Suppose that
\begin{equation}\label{eq:BadY}
y=a^{i_y}(b^{s_1})^{a^{i_1}}(b^{s_2})^{a^{i_1+i_y}}\cdots (b^{s_{2\mu-1}})^{a^{i_1}}(b^{s_{2\mu}})^{a^{i_1+i_y}}.
\end{equation}
For $u\in \{v_1,v_2\}$, we have
\[(y^p)_{u} = y_{u+i_y}y_{u+2i_y}\cdots y_{u+pi_y},\]
and since we have supposed above that $\sum_{d=0}^{k-1} |y_{u-di_y}|=\frac{|y|}{2}$ and that $|y_{i_1}|=|y_{i_1+i_y}|=\frac{|y|}{2}$, we conclude that either $u=i_1$ or $u-ki_y=i_1$. Thus, up to renaming $v_1$ and $v_2$, we can assume that $v_1=i_1+ki_y$ and $v_2=i_1$. It follows that
\begin{equation}\label{eq:yv1}
(y^p)_{v_1} = a^{k_1}y_{i_1}y_{i_1+i_y}a^{k_2},
\end{equation}
where $a^{k_1} = y_{v_1+i_y}\cdots y_{v_1+(p-k-1)i_y}$ and $a^{k_2}= y_{v_{1}+(p-k+2)i_y}\cdots y_{v_1+pi_y}$, and that
\begin{equation}\label{eq:yv2}
(y^p)_{v_2} = y_{i_1+i_y}a^{k_2+k_1}y_{i_1}.
\end{equation}

From \eqref{eq:BadY} and \eqref{eq:yv1}, we see that
\begin{align*}
(y^p)_{v_1} &= a^{k_1}b^{t_1}a^{r_1}b^{t_2}a^{r_2}\cdots a^{r_{2\mu-1}}b^{t_{2\mu}}a^{k_2}\\
&= a^{k_1+\sum_{d=1}^{2\mu-1}r_d+k_2}(b^{t_1})^{a^{\sum_{d=1}^{2\mu-1}r_d+k_2}}(b^{t_2})^{a^{\sum_{d=2}^{2\mu-1}r_d+k_2}}\cdots (b^{t_{2\mu-1}})^{a^{r_{2\mu-1}+k_2}}(b^{t_{2\mu}})^{a^{k_2}}
\end{align*}
for some $t_1,\dots, t_{2\mu}$ and $r_1, \dots, r_{2\mu-1}$ in $\mathbb{F}_p$. Suppose that $(y^p)_{v_1}$ is of the same form as $y$, namely that there exists $i_3\in X$ such that
\[
(y^p)_{v_1}=a^{i_y}(b^{t_1})^{a^{i_3}}(b^{t_2})^{a^{i_3+i_y}}\cdots (b^{t_{2\mu-1}})^{a^{i_3}}(b^{t_{2\mu}})^{a^{i_3+i_y}}.
\]
This implies that $r_d=-i_y$ if $d$ is odd and $r_d=i_y$ if $d$ is even. Therefore, we obtain $\sum_{d=1}^{2\mu-1}r_d=-i_y$. Since we know from Lemma~\ref{lem:propagates} that $k_1+k_2+\sum_{d=1}^{2\mu-1}r_d=i_y$, we conclude that $k_1+k_2=2i_y$.

If we now turn our attention to $(y^p)_{v_2}$, it follows from \eqref{eq:BadY} and \eqref{eq:yv2} that
\[(y^p)_{v_2} = a^{r'_1}b^{t'_1}\cdots a^{r'_{\mu}}b^{t'_{\mu}} a^{k_1+k_2} b^{t'_{\mu+1}}a^{r'_{\mu+1}}\cdots b^{t'_{2\mu}}a^{r'_{2\mu}}\]
for some $t'_1,\dots, t'_{2\mu}$ and $r'_1, \dots, r'_{2\mu}$ in $\mathbb{F}_p$. Using the same reasoning as above, for $(y^p)_{v_2}$ to be of the same form as $y$, we would need either $k_1+k_2=i_y$ or $k_1+k_2=-i_y$, depending on the parity of $\mu$.

We conclude that $(y^p)_{v_1}$ and $(y^p)_{v_2}$ cannot both be in the same form as $y$. Indeed, this would imply that either $i_y=2i_y$ or $-i_y=2i_y$. Since $i_y\ne 0$, the first equation is impossible, and the second can only be satisfied if $p=3$. However, we assumed that there existed some $1<k<p-1$, which is impossible if $p=3$.

Therefore, there is some $u\in \{v_1, v_2\}$ such that $y'=(y^p)_{u}$ is not in the form of \eqref{eq:BadY}. Setting $w'=(y^{-k}w)_u$ and repeating the whole argument above with $y'$ and $w'$ will thus necessarily yield a contradiction. We conclude that $|x|=1$, and the result follows.

\begin{theorem} \label{last} Let $G$
  be a GGS-group acting on the $p$-regular rooted tree, for $p$ an odd prime.  Then
  $G$ does not contain any proper prodense subgroups.
\end{theorem}

\begin{proof}
By~\cite{Pervova4}, it suffices to consider the non-torsion GGS-groups~$G$. Further we may suppose that $G$ is not conjugate to a generalised Fabrykowski-Gupta group, as otherwise the result follows by~\cite[Thm.~7.2.7]{Francoeur}.
  Suppose on the contrary that $M$ is a proper prodense  subgroup of~$G$. By Proposition~\ref{3.3.3},
  for every vertex $u\in X^*$ we have $M_u$ is properly contained in
  $G_u$. However, by Propositions~\ref{proposition: second} and~\ref{proposition: first}, there exists $v\in X^*$ such that the subgroup $M_v$ is all
  of~$G$. This gives the required contradiction.
\end{proof}

The first statement of Theorem~\ref{thm:main-result} is now proved. We show the second.

\begin{proposition}
Let $G$ be a branch GGS-group acting on the $p$-regular rooted tree, for an odd prime~$p$. Then  every maximal subgroup of~$G$ is normal and of index~$p$.
\end{proposition}

\begin{proof}
Let $M$ be a maximal subgroup of~$G$. From the previous result, it follows that $M$ has finite index in~$G$. By~\cite{FAGUA}, the group~$G$ has the congruence subgroup property, so there exists an $n\in\mathbb{N}$ such that $\st_G(n)\le M$. As $G/\st_G(n)$ is a finite $p$-group, it follows that $G'\le M\trianglelefteq_p G$. Hence the result.
\end{proof}

For the constant GGS-group~$\mathcal{G}$, the situation is different:

\begin{proposition}\label{pro:constant}
Let $\mathcal{G}$ be the weakly branch, but not branch, GGS-group acting on the $p$-regular rooted tree, for an odd prime~$p$. Then there are infinitely many maximal subgroups. In particular, the group~$\mathcal{G}$ has maximal subgroups that are neither normal, nor of index $p$.
\end{proposition}

\begin{proof}
By \cite[Prop.~3.4]{FAGUA} and using the notation of Section~3.2, we have $\mathcal{G}/K' \cong \left(\mathbb{Z}/p\mathbb{Z}\right) \ltimes \mathbb{Z}^{p-1}$ where the action of $\mathbb{Z}/p\mathbb{Z}$ on $\mathbb{Z}^{p-1}$ is given by the matrix
\[
A = 
\left( \begin{array}{@{}c|c@{}}
   \begin{matrix}
      0_{1\times p-2}
   \end{matrix} 
      & -1 \\
   \cmidrule[0.4pt]{1-2}
   I_{p-2\times p-2} & 
   \begin{matrix}
   -1\\
   \vdots\\
   -1
   \end{matrix}
\end{array} \right).
\]
We see that for all $q\in \mathbb{N}$, the subgroup $(q\mathbb{Z})^{p-1}\leq \mathbb{Z}^{p-1}$ is invariant under the action of $\mathbb{Z}/p\mathbb{Z}$. Therefore, for all $q\in \mathbb{N}$ different from $0$ or $1$, we can consider the subgroup $\left(\mathbb{Z}/p\mathbb{Z}\right) \ltimes(q\mathbb{Z})^{p-1}$, which is non-trivial and proper in $\left(\mathbb{Z}/p\mathbb{Z}\right) \ltimes \mathbb{Z}^{p-1}$. As $\left(\mathbb{Z}/p\mathbb{Z}\right) \ltimes \mathbb{Z}^{p-1}$ is finitely generated, this subgroup is contained in a maximal subgroup $M_q\leq \left(\mathbb{Z}/p\mathbb{Z}\right) \ltimes \mathbb{Z}^{p-1}$.

We notice that if $q_1, q_2\in \mathbb{N}$ are two different prime numbers, then $M_{q_1} \ne M_{q_2}$. Indeed, otherwise, we would have $(q_1\mathbb{Z})^{p-1}, (q_2\mathbb{Z})^{p-1}\leq M_{q_1}$, which would imply $\mathbb{Z}^{p-1}\leq M_{q_1}$, since $q_1$ and $q_2$ are coprime. As we already had $\left(\mathbb{Z}/p\mathbb{Z}\right) \ltimes(q_1\mathbb{Z})^{p-1}\leq M_{q_1}$, this means that $M_{q_1}=\left(\mathbb{Z}/p\mathbb{Z}\right) \ltimes \mathbb{Z}^{p-1}$, which is absurd.

We have thus shown that we have an infinite number of maximal subgroups of $\mathcal{G}/K'$ which are all pairwise distinct. By the correspondence theorem, the same is true for $\mathcal{G}$. This immediately implies that $\mathcal{G}$ has maximal subgroups that are not of index $p$, since there can only be finitely many such subgroups. It also implies that $\mathcal{G}$ admits maximal subgroups that are not normal. Indeed, otherwise, the Frattini subgroup of $\mathcal{G}$ would contain $\mathcal{G}'$, since the quotient of any group by a normal maximal subgroup must be a cyclic group of prime order, and thus abelian, but $\mathcal{G}/\mathcal{G'}$ is a $p$-group, which would imply that every maximal subgroup is of index~$p$.
\end{proof}



\begin{thebibliography}{99}

\bibitem{AKT} T.~Alexoudas, B.~Klopsch and A.~Thillaisundaram,
  Maximal subgroups of multi-edge spinal groups,
    \textit{Groups Geom. Dyn.} \textbf{10} (2016), 619--648.


\bibitem{BarthGrigSunik} L.~Bartholdi, R.\,I.~Grigorchuk and
  Z.~{\v{S}}uni{\'k}, \textit{Handbook of algebra} \textbf{3},
  North-Holland, Amsterdam, 2003.


\bibitem{Bondarenko} I.\,V.~Bondarenko, Finite generation of iterated
  wreath products, \textit{Arch. Math. (Basel)} \textbf{95 (4)} (2010),
  301--308.

\bibitem{FAGUA} G.\,A.~Fern\'{a}ndez-Alcober, A.~Garrido and J.~Uria-Albizuri, On the congruence subgroup property for GGS-groups, \textit{Proc. Amer. Math. Soc.} \textbf{145 (8)} (2017), 3311--3322.

\bibitem{FAZR2} G.\,A.~Fern\'{a}ndez-Alcober and A.~Zugadi-Reizabal, GGS-groups: Order of congruence quotients and Hausdorff dimension, \textit{Trans. Amer. Math. Soc.} \textbf{366} (2014), 1993--2017.


\bibitem{Francoeur} D.~Francoeur, \textit{On maximal subgroups and other aspects of branch groups}, PhD thesis, University of Geneva, 2019.


\bibitem{Francoeur-paper} D.~Francoeur, On maximal subgroups of infinite index in branch and weakly branch groups, \textit{J. Algebra} \textbf{560} (2020), 818--851.

 \bibitem{FG} D.~Francoeur and A.~Garrido, Maximal
  subgroups of groups of intermediate growth, \textit{Adv. Math.} \textbf{340} (2018), 1067--1107.
  
 
  

\bibitem{GUA2} A.~Garrido and J.~Uria-Albizuri, Pro-$\mathcal{C}$ congruence properties for groups of rooted tree automorphisms, \textit{Arch. Math. (Basel)} \textbf{112 (2)} (2019), 123--137.


\bibitem{Grigorchuk} R.\,I.~Grigorchuk, On Burnside's problem on
  periodic groups, \textit{Funktsional. Anal. i Prilozhen} \textbf{14 (1)}
   (1980), 53--54.

\bibitem{NewHorizons} R.\,I.~Grigorchuk, Just infinite branch groups,
  in: \textit{New horizons in pro-$p$ groups}, Birkh\"auser, Boston,
  2000.

\bibitem{GrigWils} R.\,I.~Grigorchuk and J.\,S.~Wilson, A structural
  property concerning abstract commensurability of subgroups,
  \textit{J. London Math. Soc.} \textbf{68 (2)} (2003), 671--682.

\bibitem{Gupta} N.~Gupta and S.~Sidki, On the Burnside problem for
  periodic groups, \textit{Math. Z.} \textbf{182 (3)} (1983),
  385--388.

\bibitem{KT} B.~Klopsch and A.~Thillaisundaram, Maximal subgroups and irreducible representations of generalised multi-edge spinal groups,  \textit{Proc. Edin. Math. Soc.} \textbf{61 (3)} (2018), 673--703.


\bibitem{Pervova3} E.\,L.~Pervova, Everywhere dense subgroups of a
  group of tree automorphisms, \textit{Tr. Mat. Inst. Steklova}
  \textbf{231} (Din. Sist., Avtom. i. Beskon. Gruppy) (2000),
  356--367.

\bibitem{Pervova4} E.\,L.~Pervova, Maximal subgroups of some non
  locally finite $p$-groups, \textit{Internat. J. Algebra Comput.}
  \textbf{15 (5-6)} (2005), 1129--1150.


\bibitem{Vovkivsky} T.~Vovkivsky, Infinite torsion groups arising as
  generalizations of the second Grigorchuk group, in: \textit{Algebra} (Moscow, 1998), de Gruyter, Berlin, 2000.
  

\bibitem{Zassenhaus} H.~Zassenhaus, \textit{Lehrbuch der Gruppentheorie}, Chelsea Publ. Co., New York, 1958.

\end{thebibliography}
\end{document}